\documentclass[10pt,amsfonts, epsfig]{amsart}
\usepackage{amsmath, amscd, amssymb}
\usepackage{graphpap, color}
\usepackage[mathscr]{eucal}
\usepackage{mathrsfs}
\usepackage{pstricks}
\usepackage{cancel}
\usepackage[mathscr]{eucal}
\usepackage{verbatim}
\usepackage[all]{xy}
\usepackage{stmaryrd}

\def\cA{{\cal A}}

\numberwithin{equation}{section}

\newcommand{\CC}{\mathbb{C}}

\newcommand{\PP}{\mathbb{P}}

\newcommand{\cal}{\mathcal}

\def\cM{{\cal M}}

\def\cZ{{\cal Z}}

\def\mapright#1{\,\smash{\mathop{\lra}\limits^{#1}}\,}

\def\sta{^\ast}

\def\sta{^{\ast}}

\def\sta{^*}

\def\lra{\longrightarrow}

\def\begeq{\begin{equation}}
\def\endeq{\end{equation}}
\def\and{\quad{\rm and}\quad}

\def\sub{\subset}

\def\and{\quad\text{and}\quad}

\DeclareMathOperator{\End}{End}

 \DeclareMathOperator{\rank}{rank}

\DeclareMathOperator{\Res}{Res}

\newtheorem{prop}{Proposition}[section]

\newtheorem{theo}[prop]{Theorem}

\newtheorem{coro}[prop]{Corollary}

\newtheorem{exam}[prop]{Example}

\newtheorem{defi-prop}[prop]{Definition-Proposition}

\newtheorem{quest}[prop]{Question}

\def\dbar{\overline{\partial}}

\def\sta{^\ast}

\def\beq{\begin{equation}}
\def\eeq{\end{equation}}

\def\bee{\begin{equation}}
\def\eeq{\end{equation}}

\title{Virtual residue and generalized Cayley- Bacharach Theorem}

\author[Mu-Lin Li]{Mu-Lin Li}
\address{College of Mathematics and Econometrics, Hunan University, China} \email{mulin@hnu.edu.cn}
\date{}

\begin{document}
\maketitle

\begin{abstract}
 Using virtual residue, which is a generalization of Grothendieck residue, we generalized Cayley- Bacharach Theorem to the cases with positive dimensions.
\end{abstract}
\section{introduction}

Let $C_1,C_2\subset\PP^2$ be plane curves of degrees $d$ and $e$ respectively,
meeting in a collection of $d\dot e$ distinct points $\Gamma:=\{p_1,\cdots , p_{de}\}$. Cayley- Bacharach Theorem said that if $C \subset \PP^2$ is any plane curve of degree $d+e-3$ containing all but one point of $\Gamma$, then $C$ contains all of $\Gamma$, see \cite[Theorem CB4]{EGH}. The extension of Cayley- Bacharach property on projective manifolds had been proved to be related to Fujita conjecture, and the construction of special bundles, see \cite{HS}, \cite{SLT1}, \cite{TV} and \cite{SLT2}.

By using Grothendieck residue, Griffiths and Harris \cite[Chapter5]{GH} proved the following generalized Cayley- Bacharach theorem.
Let $\cM$ be a compact complex manifold and  let $E$ be a holomorphic bundle over $\cM$ with $\rank E=\dim \cM=n$. Let $\tilde{s}$ be a  holomorphic section of  $E$ who zero loci $\cZ$ are isolated points.
If $\cZ$ consists of distinct simple points, then each $D\in |K\otimes \det E|$ that passes through all but one point of $\cZ$ necessarily contains that remaining points, where $K_{\cM}$ is the canonical line bundle of $\cM$.

In this paper we will deal with the cases where some connected components of $\cZ$ may have positive dimensions.

Let $M$ be a complex manifold and  let ${V}$ be a holomorphic bundle over $M$ with $\rank V=\dim M=n$. Let $s$ be a  holomorphic section of  $V$ with compact zero loci $Z$. Given any holomorphic section   $$\psi\in \Gamma(M,K_M\otimes \det{V}),$$
using the Koszul complex   of $(V,s)$, the authors \cite{ML1} constructed a closed form $\eta_\psi\in \Omega^{n,n-1}(M\setminus Z)$ via Griffiths-Harris's construction \cite[Chapter 5]{GH}. Then they define the virtual residue as
\beq\label{V-residue}Res_{Z}\frac{\psi}{s}:=(\frac{1}{2\pi i})^n\int_{N}\eta_\psi\in\CC\eeq
where   $N$ is a real $2n-1$ dimensional piecewise smooth compact subset of $M$ that ``surrounds $Z$", in the sense that  $N=\partial T$ for some compact domain $T\sub M$, which contains $Z$ and is homotopically equivalent to $Z$. When $\dim Z>0$, it is a generalization of Grothendieck residue. It vanishes whenever $M$ is compact by Stokes theorem.

  Denote $\cA^{i,j}(\wedge^k V\otimes\wedge^l V^*)$ to be the sheaf of smooth $(i,j)$ forms
 on $M$ valued in $\wedge^k V\otimes \wedge^lV^*$. The Hermitian metrics of $M$ and  $V$    induce a metric on the bundle which corresponds to the sheaf $\oplus_{i,j,k,l}\cA^{i,j}(\wedge^k V\otimes \wedge^lV^*)$. \black Denote this metric by $\langle\cdot,\cdot\rangle(z)$ for $z\in M$ and set $|\alpha|(z)=\sqrt{\langle\alpha,\alpha\rangle(z)}$.

 Denote $\Omega^{(i,j)}(\wedge^k V \otimes\wedge^l V^*):=\Gamma(M,\cA^{i,j}(\wedge^k V \otimes\wedge^l V^*))$ and assign its element $\alpha$ to have degree $\sharp \alpha=i+j+k-\ell$.

 Given   $u\in \Omega^{(i,j)}(\wedge^k V)$ and
 $k\geq \ell$,  we   define
 \beq\label{operator1} u \lrcorner: \Omega^{(p,q)}(\wedge^{l}V^*)\lra \Omega^{(p+i,q+j)}(\wedge^{k-l}V) \eeq
where for $\theta\in \Omega^{(p,q)}(\wedge^{\ell}V^*)$, the $u\lrcorner \theta$ is determined by
$$( u\lrcorner \theta,\nu\sta) =(-1)^{(i+j)l+(p+q)\sharp u+\frac{l(l-1)}{2}}( u,\theta\wedge \nu^*) ,\qquad \forall \nu^*\in A^0(\wedge^{k-l}V^*).$$
where $(,)$
is the dual pairing between $\wedge^k V,\wedge^k V\sta$.

Applying the integral representation for the virtual residue $\Res \frac{\psi}{s}$ (\cite[Theorem 1.1]{ML1}) to the case where $M$ is compact, we have

\begin{theo} \label{th1}Let $M$ be a compact complex manifold.
   Pick a Hermitian metric $h$ on $V$ and let $\nabla$ be its associated Hermitian connection with $\nabla^{0,1}=\dbar$. Let $\xi=-\langle*,s\rangle$ be a smooth section of $V^*$ and $$S=-|s|^2+\dbar\xi\in\oplus_{p=0,1} \Omega^{(0,p)}(\wedge^p V^*).$$
   One has
 \beq\label{diff}\Res_Z\frac{\psi}{s}=\frac{(-1)^n}{(2\pi i)^n}\int_M(\psi\lrcorner e^{S})=0.\eeq
 Here $\lrcorner$ is the operation  contracting $\det V$ with $\det V^*$ so that $\psi\lrcorner e^{S}\in\Omega^{\ast,\ast}$.
 \end{theo}

 Assuming that all the connected components $Z_i\subset Z=s^{-1}(0)$ are smooth, and $V$ is splitting over $Z_i$, $V|_{Z_i}=V_{i}\oplus N_i$, where $N_i= N_{Z_i/M}$.  Let $j: Z_i\rightarrow M$ be the embedding. Then we evaluate the integral in (\ref{diff}) as following

\begin{theo}\label{thm2}
\beq\label{formula1}
\frac{(-1)^n}{(2\pi i)^n}\int_M (\psi\lrcorner e^{-S})= \sum\frac{(-1)^n}{(2\pi i)^n}\int_{Z_i}  (\frac{\psi}{\det \ d s}\lrcorner\frac{1}{\det_{N_i}((1+R^{V_i}_s)/-2\pi \sqrt{-1})})=0,
\eeq
where $R^{V_i}_s:=-(d s)^{-1}P^{Im \ d s}R(.,j_{*})P^{V_i}\in \overline{T^*Z_i}\otimes V_i^*\otimes \End N_i$.
\end{theo}

This can be considered as a mathematical interpretation of the residue formula used in \cite[(2.11)]{WB}.

Let  $Z_1$ be one of the zero dimension components,  then we have the following generalized Cayley- Bacharach Theorem.

\begin{coro}Under the assumptions of Theorem \ref{thm2}.
If $\psi$ is vanishing on all components of $Z$ except $Z_1$, then it is vanishing on $Z$.
\end{coro}

\begin{quest} Can we extend the Cayley- Bacharach Theorem to the cases where all the connect components of $Z$ are positive dimensions?
\end{quest}
\noindent{\bf Acknowledgment}:  The author thanks Hao Sun for informing him the relations between Cayley- Bacharach property and Fujita conjecture. This work was supported by Start-up Fund of Hunan University.
\section{Localization by section}

Let $Z=\cup Z_i$, and $Z_i$ be smooth connected component. Denote by $N_i=N_{Z_i/M}$ the normal bundle of $Z_i\subset M$.  Because $Z_i$ is smooth, the Kuranishi sequence gives us a exact sequence

$$
0\lra T_{Z_i} \lra T_M|_{Z_i}\mapright{d s} V|_{Z_i} \lra V_i\lra 0.
$$
Since $N_i\cong T_M|_{Z_i}/T_{Z_i}$, the above sequence gives us the following short exact sequce
$$
0\lra N_i\mapright{d s} V|_{Z_i} \lra V_i\lra 0.
$$
Assuming that the above exact sequence is splitting, therefore $V|_{Z_i}=V_i\oplus N_{i}\cong V_i\oplus Im\ ds$. Let $\psi\in\Gamma(M,K_M\otimes\det V)$, thus it can be viewed as morphism
\beq\label{equ1}\psi :\det V^{*}\rightarrow K_M.\eeq
 $d s $ induced the following isomorphism
 \beq\label{equ2}\det d s: \det N_i\rightarrow \det Im\  ds.\eeq
 (\ref{equ1}) and (\ref{equ2}) induced a morphism
 $$\frac{\psi}{\det \ d s}: \det V_i^{*}\rightarrow K_{Z_i}.$$
 The correspondent element in $\Gamma(Z,K_{Z_i}\otimes\det V_i)$ is also denoted by $\frac{\psi}{\det \ d s}$.

Let $h$ be a Hermitian metric on V such that $V_i$ and $N_i$ are orthogonal on $Z_i$ .  Let $g_i$ be a Hermitian metric on $N_i$ such that $ds: N_i \cong Im\ ds$ is an isometry. Let $R^{V}$ be the curvature of the holomorphic Hermitian connection $\nabla$ on $(V,h)$. Let $j: Z_i\rightarrow M$ be the embedding. Let $P^{V_i}$ and $P^{Im \ d s}$ be the natural projections from $V$ onto $V_i$ and $Im\ d s$ . Let

 $$R^{V_i}_s:=-(d s)^{-1}P^{Im \ d s}R^{V}(.,j_{*}.)P^{V_i}\in \overline{T^*Z_i}\otimes V_{i}^*\otimes \End N_i$$
$R^{V_i}_s$ is well defined since $P^{Im \ d s} R^{V}( j_*., j_{*}.)P^{V_i}$ = 0.
\begin{theo}\label{th4}
 Under the above assumptions we have the following formula,
\beq\label{formula}
\int_M (\psi\lrcorner e^{-S})= \sum\int_{Z_i}  (\frac{\psi}{\det \ d s}\lrcorner\frac{1}{\det_{N_i}((1+R^{V_i}_s)/-2\pi \sqrt{-1})})=0.
\eeq
\end{theo}
\begin{proof}{By \cite[Proposition 4.14]{ML1},
 $\int_M (\psi\lrcorner e^{(2t)^{-1}S})$ is  independent of $t$ for $t>0$.
 Therefore we can do the calculation by letting $t\to 0$. This method is parallel to the one used in \cite{Feng}.
For arbitrary $y\in Z_i$, since $Z_i$ is a complex submanifold with dimension $m_i$, we can find out holomorphic coordinates $\{z_i\}$ of the neighborhood $U$ of $y$ such that $y$ corresponds to 0, and  $\{\frac{\partial}{\partial z_i}\}_{i=m_i+1}^n$ is an orthonormal basis of the normal bundle $N_y$. Moreover $U\cap Z_i=\{p\in U,z_{m_i+1}(p)=\cdots=z_n(p)=0\}$. Denote $z^{\prime}=(z_1,\cdots,z_{m_i}), z^{\prime\prime}=(z_{m_i+1},\cdots,z_n), z=(z^{\prime},z^{\prime\prime})$.

Let $\{\mu_k(z^{\prime},0)\}_{k=1}^n$ and $\{\mu_k(z^{\prime},0)\}_{k=m_i+1}^n$ be the holomorphic frame for $V$ and $Im\ d s$ on $U\cap Z_i$ with

\begin{eqnarray*}
\nabla_{\frac{\partial}{\partial z_k}} s |_y&=&\frac{\partial s}{\partial z_k}|_y= \mu_k(0)\ \ \ \ \ \   m_i+1\le k\le n.
\end{eqnarray*}
Let $\{\mu^k(z^{\prime},0)\}_{k=1}^n$  be the correspondent basis of $V^{*}$.
 Define $\mu_k(z)$ by parallel transport of $\mu_k(z^{\prime},0)$ with respect to $\nabla$ along the curve $u\rightarrow (z^{\prime},uz^{\prime\prime})$. Identify $V_z$ with $V_{(z^{\prime},0)}$ by identify $\mu_k(z)$ with $\mu_k(z^{\prime},0)$. Denote by $ W_y(\epsilon)$ the neighborhood of $y$ in the normal space $N_i$. Then

\begin{eqnarray}\label{equ3}
&&\int_{Z_i\cap U}\int_{W_y(\epsilon)}(\psi \lrcorner e^{-\frac{1}{2t}(\dbar (-\xi)+ |s|^2 )})\\
&=&\int_{Z_i\cap U}\int_{W_y(\epsilon)/\sqrt{t}}t^{n- m_i}(\psi(y,\sqrt{t}z)\lrcorner e^{-\frac{1}{2t}(\dbar (-\xi)(\sqrt{t}z)+ |s(\sqrt{t}z)|^2 )}).\nonumber
\end{eqnarray}


From now on we set $z=(0,z'')$, $v_z=\sum_{j=m_i}^n z_j(\frac{\partial}{\partial z_j})$ and $Y=v_z+\bar{v}_z$. The tautological vector field is $Y=v_z+\bar{v}_z$. Then
$$
\frac{1}{2t}|s(\sqrt{t}z)|^2 = \frac{1}{2}|\nabla_Y s|^2+\mbox{O}(t^{\frac{1}{2}}) =\frac{1}{2}|z|^2+O(\sqrt{t}),
$$
and

$$
\dbar \langle*,s\rangle=\sum_{k=1}^n\langle\mu_k,\nabla s\rangle\mu^k.
$$

 Since $\nabla_Y \mu_k(0)=0$, we have
\begin{eqnarray*}
\frac{1}{2t}\dbar \langle*,s\rangle(\sqrt{t}z)&=&\frac{1}{2t}\sum_{k=1}^n \langle\mu_k,\nabla s\rangle(\sqrt{t}z)\mu^k(0)\\
&=&\frac{1}{2t}\sum_{k=1}^n \bigg(\langle\mu_k,\nabla s\rangle(0)+\sqrt{t}\langle\mu_k,\nabla_Y \nabla s\rangle(0)\\
&& +\frac{t}{2}(\langle\nabla_Y\nabla_Y\mu_k,\nabla s\rangle+\langle\mu_k,\nabla_Y\nabla_Y\nabla s\rangle)(0)+O(t^{\frac{3}{2}})\bigg)\mu^k(0).
\end{eqnarray*}

Because there is a factor $t^{n-m_i}$ in (\ref{equ3})\black, it should be clear that in the limit, only those monomials in the vertical form
$$
d\overline{z_{m_i+1}}\wedge\cdots\wedge d\overline{z_n}\otimes \mu^{m_i+1}(0)\wedge\cdots\wedge\mu^{n}(0)
$$
whose weight is exactly $t^{m_i-n}$ should be kept.

So the second term contribute zero to the integral, and the terms contributes nonzero in the third term are

$$
\frac{1}{4}\sum_{k=1}^{n}\sum_{j=1}^{m_i} (\langle\nabla_Y\nabla_Y\mu_k,\nabla_{\frac{\partial}{\partial z_j}} s\rangle(0)+\langle\mu_k,\nabla_Y\nabla_Y\nabla_{\frac{\partial}{\partial z_j}}s\rangle(0))d \bar{z}_j\otimes\mu^k(0).
$$

But for $1\le j\le m_i$, both $\nabla_{\frac{\partial}{\partial z_j}}s(0)=0, \nabla_{\frac{\partial}{\partial z_j}}\nabla_{\bar{v}_z}\nabla_{v_z}s(0)=\nabla_{\frac{\partial}{\partial z_j}}(R^{V}(\bar{v}_z,v_z)s)(0)=0$. Thus
$$
\nabla_{Y}\nabla_{Y}\nabla_{\frac{\partial}{\partial z_j}}s(0)=2R^{V}(\bar{v}_z,\frac{\partial}{\partial z_j})\nabla_{v_z}s(0)+\nabla_{\frac{\partial}{\partial z_j}}\nabla_{v_z}\nabla_{v_z}s(0).
$$

Note that $\nabla=\nabla^{N_i}\oplus\nabla^{V_i}$ on $Z_i$, where $\nabla^{N_i}$ and $\nabla^{V_i}$ are induced connections. By previous discussion, as $t\to 0$, we should replace $\frac{1}{2t}(\dbar (-\xi)+ |s|^2 )$ by,
\begin{eqnarray*}
&&\sum_{k=m_i+1}^n \frac{1}{2t}\langle\mu_k,\nabla s\rangle(0)\mu^k(0)\\
&&+\frac{1}{2}\sum_{k=1}^{m_i}\sum_{j=1}^{m_i}\big<\mu_k,R^{V}(\bar{v}_z,\frac{\partial}{\partial z_j})\nabla_{v_z} s\big>(0)d\bar{z_j}\otimes \mu^k(0).
\end{eqnarray*}

For $\int_{\mathbb{C}}\bar{z}^ie^{-|z|^2}dz\wedge d\bar{z}=0$. So as $ t\to 0$, the integral becomes:
$$
\int_{Z_i}\int_{N_i} (\psi\lrcorner\mbox{exp}(-\frac{1}{2}\sum\langle\mu_k,\nabla s\rangle(0)\mu^k(0)-\frac{1}{2}\langle\cdot,P^{V_i}R^{V}(\bar{v}_z,j_{*}\cdot)\nabla_{v_z} s\rangle(0)-\frac{1}{2}|\nabla_ {v_z} s|^2)).
$$

The second integrand is equal to
\begin{eqnarray*}
&&\mbox{exp}(-\frac{1}{2}\sum_{k=m_i+1}^n d\bar{z}_{k}\wedge\mu^{k}(0)+\frac{1}{2}\langle R^{V}(v_z,j_{*}\cdot)P^{V_i},\nabla_{v_z} s\rangle(0)-\frac{1}{2}|z|^2)\\
&=&\mbox{exp}(\frac{1}{2}\langle R^{V}(v_z,j_{*}\cdot)P^{V_i},\nabla_{v_z} s\rangle-\frac{1}{2}|z|^2)(\frac{1}{2})^{n-m_i}d z_{m_i+1}\wedge \cdots dz_{n}\otimes \mu^{m_i+1}(0)\wedge\cdots\wedge \mu^n(0).
\end{eqnarray*}

So the whole integral is equal to
$$
\int_{Z_i} (\frac{\psi}{\det d s}\lrcorner\frac{1}{\det_{N_i}((1+R^{V_i}_s)/-2\pi \sqrt{-1})}),
$$
where $P^{V_i}$ is the projection from $V$ to $V_i$, $P^{Im \ d s}$ is the projection from $V$ to $Im \ d s$ and $R^{V_i}_s:=-(d s)^{-1}P^{Im \ d s}R^{V}(.,j_{*})P^{V_i}\in \overline{T^*Z_i}\otimes V_i^*\otimes \mbox{End}N_i$.
}
\end{proof}

\begin{exam}Let M be a compact complex manifold with $\dim M=n$, $L$ be a holomorphic bundle with rank $r<n$. Let $s\in\Gamma(M,L)$ be transversal. Then the zero locus $Z$ of $s$ is smooth with $\dim Z=n-r$. Let $V=L\oplus V_1$, where $V_1$ is a holomorphic bundle with rank $n-r$. $s$ can be considered as a section of $V$. It satisfies all the condition of the above theorem.
\end{exam}

Let $V$ be a holomorphic bundle over a compact complex manifold $M$, with rank $V=\dim M=n$. Let $\psi\in\Gamma(M,K_M\otimes\det V)$, and $s\in\Gamma(M,V)$ be a transversal section with smooth zero loci.  Then the zero locus of $s$ are finite points $\{p_i\}$. Assuming that around a neighborhood of $p_i$, $s=\sum s_k e_k$, and $\psi=h(z)dz_1\wedge\cdots\wedge dz_n\otimes e_1\wedge\cdots\wedge e_n$. Then we have the following equalities, which recovered the residue theorem in \cite[Page 731]{GH}.
\begin{coro}\label{classical prop}
$\Res_{Z}\frac{\psi}{s}=\sum_{p_i}\frac{h(p_i)}{\det(\frac{\partial s}{\partial z}(p_i))}=0$.
\end{coro}

Let $M$ be a compact manifold, $V$ be a holomorphic bundle over $M$ that $\mbox{rk} V=\dim M$, with a section $s\in \Gamma(M,V)$ and the zero loci $Z=\cup_{i=1}^w Z_i$, where all $Z_i$ are smooth and at least one $Z_i$ is zero dimension.  $V$ is splitting on $Z_i$ as in Theorem \ref{th4}. Let $\psi$ be a section of $ K_M\otimes \det V$. Then we have the following general Cayley- Bacharach theorem.
\begin{coro}
With the assumptions as above and $\psi$ is vanishing on all components of $Z$ except one of the zero dimension component $Z_i$, then it is vanishing on whole $Z$.
\end{coro}
\begin{proof}
Assuming that one of $Z_i$ is point $p$, then by the corollary \ref{classical prop}, $\Res_p\frac{\psi}{s}=\frac{h(p)}{\det(\frac{\partial s}{\partial z}(p))}$. For $M$ is compact, we have
$$\Res_{Z}\frac{\psi}{s}=\frac{(-1)^n}{(2\pi i)^n}\sum\int_{Z_i}  \bigg(\frac{\psi}{\det \ d s}\lrcorner\frac{1}{\det_{N_i}((1+R^{V_i}_s)/-2\pi \sqrt{-1})}\bigg)=\frac{h(p)}{\det(\frac{\partial s}{\partial z}(p))}=0.$$
So $h(p)=0$.
\end{proof}

\bibliographystyle{amsplain}

\end{document}